\documentclass[11pt, oneside]{article}
\usepackage[utf8]{inputenc}
\usepackage{amsmath, amsthm, amssymb}
\usepackage[dvipsnames]{xcolor}
\usepackage{enumerate, comment, multirow, parskip}
\usepackage{dsfont}
\usepackage{enumitem}
\usepackage{mathtools}
\usepackage{indentfirst}
\usepackage{amsthm}

\usepackage{comment}
\usepackage{makecell}

\usepackage{subcaption}
\usepackage{todonotes}
\usepackage{multicol}
\usepackage{tikz}
\usetikzlibrary{decorations.markings}
\usetikzlibrary{decorations.pathmorphing}
\usetikzlibrary{calc,backgrounds}
\usetikzlibrary{patterns}
\usepackage{standalone}
\usepackage{mathtools,mdframed}
\usepackage{adjustbox}
\usepackage[left=1in,right=1in,top=1in,bottom=1in,bindingoffset=0cm]{geometry}

\usepackage[normalem]{ulem}

\usepackage{hyperref}
\hypersetup{colorlinks=true,linkcolor=blue,anchorcolor=blue,citecolor=green,filecolor=blue,urlcolor=blue,bookmarksnumbered=true,pdfview=FitB}

\usepackage[capitalise]{cleveref}

\usepackage{tabularx}
\usepackage{booktabs}

\definecolor{green}{RGB}{50,205,50}
\definecolor{yellow}{RGB}{240,240,10}

\newtheorem{theorem}{Theorem}[section]
\newtheorem{lemma}[theorem]{Lemma}
\newtheorem{corollary}[theorem]{Corollary}

\newtheorem{conjecture}[theorem]{Conjecture}

\newtheorem{question}[theorem]{Question}

\newtheorem{proposition}[theorem]{Proposition}
\newtheorem*{theorem*}{Theorem}
\newtheorem*{corollary*}{Corollary}

\theoremstyle{definition}

\newtheorem{remark}{Remark}

\theoremstyle{definition}

\title{Frozen colorings and recolorability of blowups}

\author{
Tianyu Huang\footnote{Center for Combinatorics and LPMC, Nankai University, Tianjin 300071, P.R.China. Email: tyhuang@mail.nankai.edu.cn.}
\and 
Yulai Ma\footnote{Center for Combinatorics and LPMC, Nankai University, Tianjin 300071, P.R.China. Email:  ylma@nankai.edu.cn.}
\and
Jiaqi Zhang\footnote{Center for Combinatorics and LPMC, Nankai University, Tianjin 300071, P.R.China. Email: jiaqizhang@mail.nankai.edu.cn.}}

\begin{document}
\date{}
\maketitle

\begin{abstract}
The \emph{$k$-reconfiguration graph} $R_k(G)$ of the proper $k$-colorings of a graph $G$
is the graph whose vertices are the $k$-colorings of $G$, with an edge between
two colorings if they differ on exactly one vertex. A graph $G$ is
\emph{recolorable} if $R_k(G)$ is connected for all $k \ge \chi(G) + 1$.

In this paper, we study recolorability under blowups through frozen colorings, here blowup means replacing vertices with non-empty cliques. We characterize frozen colorings of a blowup of $G$ in terms of maximal independent sets of the base graph, and develop a lifting principle that transfers suitable frozen obstructions from an induced subgraph to a blowup of the ambient graph. This framework yields several consequences. In particular, we answer negatively a question of Belavadi, Cameron and Sintiari on recoloring via modular decomposition, and disprove their conjecture on blowups of prime graphs.

We further determine exactly which cycles have all blowups recolorable: this holds for $C_n$ if and only if $n\in\{3,4,5,7\}$.
The frozen obstructions arising from the negative cases, together with the lifting principle, imply that if every blowup of a graph $G$ is recolorable, then every induced hole of $G$ has length $4$, $5$, or $7$. Finally, using linear programming duality, we show that every such graph $G$ is fractionally co-strongly perfect.

\end{abstract}

\section{Introduction}
Graph recoloring is a fundamental reconfiguration problem in which one studies
how proper colorings of a graph can be transformed into one another by changing
the color of a single vertex at a time. For an integer $k\ge 1$, the
\emph{$k$-reconfiguration graph} $R_k(G)$ has the proper $k$-colorings of
$G$ as its vertices, with two colorings adjacent when they differ on exactly
one vertex. A graph $G$ is \emph{$k$-mixing} if $R_k(G)$ is connected, and
is \emph{recolorable} if it is $k$-mixing for every
$k\ge \chi(G)+1$.

A classical starting point in graph recoloring is degeneracy. 
A graph $G$ is \emph{$d$-degenerate} if every nonempty subgraph of $G$
contains a vertex of degree at most $d$.
Dyer et al.~\cite{dyer2006randomly} and Cereceda et al.~\cite{cereceda2009mixing}
independently proved that every $d$-degenerate graph is $k$-mixing for
$k\ge d+2$. Cereceda~\cite{cereceda2007mixing} further conjectured that, in this
range, the diameter of the recoloring graph is $O(n^2)$.
This conjecture remains open in full generality, but substantial progress has
been made. In particular, Bousquet and Heinrich~\cite{bousquet2022polynomial} proved
that the diameter is $O(n^{d+1})$ for $k\ge d+2$, and is quadratic when
$k\ge \frac{3}{2}(d+1)$, while Bousquet and
Perarnau~\cite{bousquet2016fast} obtained a linear bound when $k\ge 2d+2$.
Related bounds for sparse graphs in terms of maximum average degree were
obtained by Feghali~\cite{feghali2021reconfiguring}.

Recoloring has also been studied extensively for structurally restricted graph
classes. Results are known, for example, for chordal and chordal bipartite
graphs~\cite{bonamy2014reconfiguration,merkel2022recolouring}, graphs of bounded treewidth via tree
decompositions~\cite{bonamy2013recoloring,bonamy2018recoloring, wang2026linear}, planar graphs~\cite{bartier2023recoloring, dvovrak2021thomassen} and related sparse
classes~\cite{Bonamy_Bousquet_Perarnau_2021,bousquet2016fast,cranston2022list}, and several hereditary classes defined by
forbidden induced subgraphs~\cite{belavadi2024reconfiguration,lei2024reconfiguration}.
These developments motivate the use of graph decompositions and hereditary
structure in recoloring problems.

Belavadi, Cameron and Sintiari~\cite{belavadi2026recoloring} recently studied recoloring through
modular decomposition. A nonempty set $S\subseteq V(G)$ is a \emph{module} if every vertex in $V(G)\setminus S$ is either adjacent to every vertex of $S$ or no vertex of $S$. A module $S\subset V(G)$ is said to be \emph{nontrivial} if $S$ contains at least two vertices. A graph is \emph{prime} if it has no
nontrivial module. Note that a \emph{blowup} of a graph $G$ is obtained
by replacing each vertex of $G$ by a nonempty clique, with two replacement
cliques made complete or anticomplete according to whether the corresponding
vertices of $G$ are adjacent or nonadjacent. They proved that if every
blowup of every prime graph in a hereditary class $\mathcal G$ is recolorable,
then every graph in $\mathcal G$ is recolorable. They then asked whether the
blowup assumption can be omitted.  

\begin{question}[\cite{belavadi2026recoloring}]\label{Question-BCS}
If every prime graph in a hereditary class is recolorable, must every graph in the class be recolorable?
\end{question}

For prime graphs, they also proposed the following conjecture.

\begin{conjecture}[\cite{belavadi2026recoloring}]\label{Conjecture-BCS}
Let $G$ be a prime graph. Every blowup of $G$ is recolorable if and only if every induced subgraph of $G$ is recolorable.
\end{conjecture}

In this paper, we answer Question \ref{Question-BCS} negatively and disprove Conjecture \ref{Conjecture-BCS}. 

\begin{theorem}\label{coro1.4}
\begin{enumerate}
\item[\textup{(i)}]There are infinitely many hereditary classes in which every prime
graph is recolorable but some graph in the class is not recolorable.
\item[\textup{(ii)}]There are infinitely many prime graphs $G$ such that every induced subgraph of
$G$ is recolorable but some blowup of $G$ is not recolorable.
\end{enumerate}
\end{theorem}


The counterexamples in Theorem~\ref{coro1.4} are obtained from a characterization of frozen colorings of blowups in terms of maximal independent sets. We next investigate structural restrictions on graphs whose every blowup is recolorable. A frozen coloring becomes an obstruction to recolorability when the same blowup admits a coloring with fewer colors. Our main technical tool shows that such an obstruction can be lifted from an induced subgraph to the whole graph.

\begin{theorem}\label{thm:lifting}
Let $H$ be a nonempty induced subgraph of a graph $G$. Suppose that some 
blowup of $H$ admits a frozen $r$-coloring and also a proper coloring with at
most $r-1$ colors. Then there exist a  blowup $G(w)$ of $G$ and an
integer $\ell$ such that
$\chi(G(w))\le \ell-1$ and
 $R_\ell(G(w))$ is disconnected.
In particular, $G$ has a non-recolorable  blowup.
\end{theorem}

The proof enlarges the original frozen configuration to a permanently locked
induced core in a suitable blowup of $G$. Thus an obstruction found inside an
induced subgraph survives in an appropriate blowup of the ambient graph. This lifting principle will be used below to derive structural restrictions on graphs whose every blowup is recolorable.


We next turn to structural consequences of Theorem \ref{thm:lifting}. Applying it to frozen obstructions supported on induced cycles gives the following restriction.

\begin{theorem}\label{Induced-hole restriction}
If every blowup of a graph $G$ is recolorable, then every induced hole
of $G$ has length $4$, $5$, or $7$.
\end{theorem}

It is natural to identify the correctness of the converse to establish a stronger bidirectional equivalence relation, while there exists a simple counterexample for $C_7$. At the level of cycles, this restriction is sharp. In fact, we obtain the
following exact classification.

\begin{theorem}\label{thm:3457}
Every blowup of $C_n$ is recolorable if and only if $n\in\{3,4,5,7\}$.
\end{theorem}

For $n\ge4$, the only-if direction follows from Theorem \ref{Induced-hole restriction}. For the other side, the case $C_3$ is immediate, while the cases $C_4$ and $C_5$ follow from the result of Belavadi, Cameron and Sintiari \cite{belavadi2026recoloring} that every blowup of a $3K_1$-free graph is recolorable. The remaining case $C_7$ requires a different
positive argument. The proof proceeds by induction on the number of vertices: colorings with an unused or singleton color are recolored to a fixed optimal coloring via the induction hypothesis, while all remaining colorings are shown to reach a singleton class in at most three moves. The argument rests on a structural property of $C_7$, the uniqueness of a maximum independent set avoiding an edge, which reduces the analysis to a short finite case check.


The induced-hole restriction is not the only consequence of the frozen-core framework. A second, rather different structural restriction arises from a connection with fractionally co-strongly perfect graphs, introduced by Chudnovsky, Ries and
Zwols~\cite{Chudnovsky_1,Chudnovsky_2}. For a nonempty graph $H$, a
\emph{saturating vertex weighting} is a function
$y:V(H)\to\mathbb R_{\ge 0}$ satisfying
\[
 \sum_{v\in I} y(v)=1
\]
for every maximal independent set $I$ of $H$. A graph is
\emph{fractionally co-strongly perfect} if every nonempty induced subgraph has
a saturating vertex weighting. Using linear programming duality, we
characterize the existence of such a weighting in terms of frozen colorings of
blowups. Combined with Theorem~\ref{thm:lifting}, this gives the
following.

\begin{theorem}\label{fractionally
co-strongly perfect}
If every blowup of a graph $G$ is recolorable, then $G$ is fractionally
co-strongly perfect.
\end{theorem}

Theorem~\ref{fractionally
co-strongly perfect} also gives another perspective on
Theorem~\ref{Induced-hole restriction}. Chudnovsky, Ries and
Zwols~\cite{Chudnovsky_1,Chudnovsky_2} characterized fractionally co-strongly perfect
claw-free graphs by forbidden induced subgraphs; their forbidden family
includes $C_6$ and every $C_k$ with $k\ge8$. Consequently,
Theorem~\ref{fractionally co-strongly perfect} yields an alternative derivation of the
induced-hole restriction. We nevertheless give a direct proof of
Theorem~\ref{Induced-hole restriction}, since it produces explicit frozen
obstructions and illustrates the lifting mechanism.

The remainder of the paper is organized as follows.
Section~2 contains preliminaries and the maximal-independent-set
characterization of frozen colorings in blowups.
Section~3 constructs the counterexamples to Question~\ref{Question-BCS}
and Conjecture~\ref{Conjecture-BCS}.
Section~4 proves Theorem \ref{thm:lifting}.
Section~5 establishes the induced-hole restriction.
Section~6 completes the classification of blowups of cycles by proving the
positive case $C_7$.
Section~7 develops the connection with fractionally co-strongly perfect
graphs, and Section~8 contains concluding remarks.

\section{Preliminaries}\label{section2}
 All graphs considered in this paper are finite and simple. Unless otherwise stated, we follow the standard notation an terminology(see \cite{bondy2008graph}). Two vertices $u$ and $v$ are \textit{adjacent} in $G$ if $uv\in E(G)$. A \textit{$k$-coloring} is a mapping $c: V(G)\rightarrow \{1,2,\dots,k\}$ such that any two adjacent vertices of $G$ receive different colors. If $G$ admits a $k$-coloring, then $G$ is \textit{$k$-colorable}. The minimum $k$ for which a graph is $k$-colorable is called its \textit{chromatic number}, denoted by $\chi(G)$. For a graph $G$ and $X\subseteq V(G)$, the subgraph \textit{induced} by $X$ is denoted by $G[X]$. An induced subgraph $H$ of $G$ is said to be \textit{proper} if $V(H)\subset V(G)$. Particularly, the vertex subset $X$ of $G$ is called an \textit{independent set} if $G[X]$ contains no edges, while it is called a \textit{clique} if $G[X]$ is a complete graph. For a graph $G$, the \textit{open neighborhood} of a vertex $x$, denoted by $N_{G}(x)$, is the set of vertices that are adjacent to $x$ in $G$. The \textit{closed neighborhood} of a vertex $x$, denoted by $N_{G}[x]$, is $N_{G}(x)\cup \{x\}$.
A $k$-coloring $\alpha$ of a graph $G$ is \textit{frozen} if, for each vertex $v$, every color occurs in $N_{G}[v]$. Note that a frozen coloring is an isolated vertex in $R_{k}(G)$.

In what follows, we list some lemmas that will be utilized in the later sections. Let $G(w)$ be a blowup of $G$, where each vertex $v$ is replaced by a clique $Q_v$ of size $w(v)$. For a coloring $\alpha$ of $G(w)$ and a color $i$, define
$$I_i(\alpha)=\{v\in V(G):\mbox{some vertex of $Q_v$ has color $i$}\}.$$

\begin{lemma}\label{frozen}
A proper coloring $\alpha$ of $G(w)$ is frozen if and only if every $I_i(\alpha)$ is a maximal independent set of $G$.
\end{lemma} 

\begin{proof}
It is straightforward to check that $I_i(\alpha)$ is an independent set since $\alpha$ is proper. Suppose first that $\alpha$ is frozen. For any vertex $v$ of $G$, choose a vertex $x\in Q_v$. Since $\alpha$ is frozen, every color occurs in the closed neighborhood of $x$. Hence for each color $i$, either $v\in I_i(\alpha)$ or some neighbor of $v$ belongs to $I_i(\alpha)$. Therefore $I_i(\alpha)$ is a dominating independent set, and hence is a maximal independent set. 

Conversely, assume that every $I_i(\alpha)$ is a maximal independent set. For any vertex $v$ of $G$, choose a vertex $x\in Q_v$ and fix a color $i$. If $v\in I_i(\alpha)$, then color $i$ appears in $Q_v$. Otherwise, maximality gives a neighbor $u$ of $v$ in $I_i(\alpha)$, and color $i$ appears in $Q_u$. Thus every color appears in $N[x]$, so no vertex can be recolored.
\end{proof}

\begin{remark}
It is well known that a proper coloring is frozen if and only if each of its color classes is an independent dominating set; in the coloring literature such colorings are also called \emph{fall colorings}. See, for example, \cite{laskar2009fall}. Lemma~\ref{frozen} is the corresponding
formulation for blowups, expressed in terms of maximal independent sets~(for convenience, we shortly write it as MISs) of the base graph.
\end{remark}

\begin{corollary}\label{coro_frozen}
Let $I_1,\dots, I_k$ be maximal independent sets of graph $G$ whose union is $V(G)$. Then there exists a blow up of $G$ with a frozen $k$-coloring.
\end{corollary}
\begin{proof}
First, we construct a blowup of $G$ by replacing each vertex $v$ with a clique of size 
$|\{i:v\in I_i\}|$.
Then, color the copies of $v$ by precisely those $i$ for which $v\in I_i$. According to Lemma~\ref{frozen}, one may check that it is a frozen $k$-coloring.
\end{proof}

\begin{lemma}\label{lem:locking}
Let $F$ be an induced subgraph of $B$, and let $\alpha$ be a proper
$\ell$-coloring of $B$ whose restriction to $F$ is frozen with respect to all
$\ell$ colors. Then every coloring reachable from $\alpha$ in $R_{\ell}(B)$
agrees with $\alpha$ on $F$.
\end{lemma}

\begin{proof}
Suppose a recoloring sequence changes a vertex of $F$, and consider its first
such move. Until that move, the coloring of $F$ is unchanged. Every vertex of
$F$ still sees all $\ell$ colors inside its closed neighborhood in $F$, so the
proposed move is impossible. Recoloring vertices outside $F$ cannot remove
these internal obstructions.
\end{proof}
We also need a lemma from Dyer et al.~\cite{dyer2006randomly}.
\begin{lemma}[\cite{dyer2006randomly}]\label{lemma3.4}
    For a $d$-degenerate graph $G$, $R_{k}(G)$ is connected when $k\geq d + 2$.
\end{lemma}

\section{Proof of Theorem~\ref{coro1.4}}\label{counterC1}
In this section, we construct the counterexamples to Question \ref{Question-BCS} and Conjecture~\ref{Conjecture-BCS}. Let $G_m=C_{3m}$ for odd $m\geq 3$. Write $G_m=v_0v_1\cdots v_{3m-1}v_0$. Define a blowup $H_m$ of $G_m$ by replacing every vertex $v_{3j}$ with a clique of size two and leaving all other replacement cliques of size one. Let 
$$A_r=\{v_{3j+r}: 0\leq j< m\}, \qquad r=0,1,2.$$
Each $A_r$ is a maximal independent set of $G_m$. Applying Lemma~\ref{frozen} to 
$A_0, A_0, A_1, A_2$ gives a frozen 4-coloring of $H_m$.

\begin{lemma}\label{$H_m$}
For every odd $m\geq 3$, the graph $H_m$ satisfies $\chi(H_m)=3$ and $H_m$ is not recolorable.
\end{lemma}
\begin{proof}
The frozen coloring above shows that $R_4(H_m)$ contains an isolated vertex. It remains to prove $\chi(H_m)=3$. 

Clearly, we have  $\chi(H_m)\geq3$ since $H_m$ contains triangle. Because $m$ is odd, $C_m$ has a proper 3-coloring $c_0,c_1,\dots, c_{m-1}$. For every $j$, color the two vertices in the replacement clique of $v_{3j}$ with the two colors different from $c_j$. Color the vertices replacing $v_{3j+1}$ and $v_{3j+2}$ with $c_j$ and $c_{j+1}$, respectively. Consecutive replacement cliques receive disjoint color sets, and hence this is a proper 3-coloring of $H_m$. Therefore $\chi(H_m)=3$. Since $H_m$ has a frozen 4-coloring and $4=\chi(H_m)+1$, the graph $R_4(H_m)$ is disconnected.
\end{proof}

We have verified that this example has a blowup that is not recolorable. It remains to prove that $C_{3m}~(\text{for odd }m\geq 3)$ is prime and all of its induced subgraphs are recolorable.

\begin{lemma}\label{C_nprime}
    For every odd $n\geq 5$, $C_n$ is prime and every induced subgraph of $C_n$ is recolorable.
\end{lemma}
\begin{proof}
    If $M$ is a nontrivial module of $C_n$, then a vertex outside $M$ adjacent to $M$ must be adjacent to all vertices of $M$. Since $C_n$ has maximum degree two, $|M|\leq 2$. But $|M|\geq 2$, which means $|M|=2$. The two vertices of $M$ have distinct closed neighborhoods, whether or not they are adjacent, which leads to a contradiction.

    The cycle $C_n$ is 2-degenerate and has chromatic number three, hence its recoloring graph is connected for every $k\geq 4$. Proper induced subgraphs of $C_n$ are forests, hence recolorable by Lemma ~\ref{lemma3.4}.
\end{proof}

\begin{proof}[\bf Proof of Theorem \ref{coro1.4}(i)]

For odd $m\geq 3$, let $\mathcal{C}_m$ be the class of all induced subgraphs of all blowups of $C_{3m}$. Then the non-recolorable graph $H_m$ belongs to $\mathcal{C}_m$. We show that every prime graph in $\mathcal{C}_m$ is recolorable.
    
Let $F\in \mathcal{C}_m$ be prime. By definition, $F$ is an induced subgraph of some blowup of $C_{3m}$. In any such blowup, each replacement clique is a module. If $F$ intersects two or more replacement cliques, then for each such intersection $F \cap Q_v$, this set is a module of $F$. Since $F$ is prime, every module of $F$ is trivial; hence $\lvert F \cap Q_v \rvert \le 1$ for every vertex $v$ of $C_{3m}$. Therefore $F$ contains at most one vertex from each replacement clique, so $F$ is isomorphic to an induced subgraph of $C_{3m}$. By Lemma \ref{C_nprime}, every induced subgraph of $C_{3m}$ is recolorable (since $3m$ is odd and $3m \ge 9$). Thus $F$ is recolorable. On the other hand, if $F$ intersects only a single replacement clique, then $F$ is a complete graph, and hence is recolorable.
   
Thus every prime graph in $\mathcal{C}_m$ is recolorable, but $\mathcal{C}_m$ contains a non-recolorable graph (namely $H_m$). As a consequence, there are infinitely many hereditary classes in which every prime graph is recolorable but some graph in the class is not recolorable.
\end{proof}

\begin{proof}[\bf Proof of Theorem \ref{coro1.4}(ii)]
By Lemma \ref{C_nprime}, for odd $m\geq 3$, $C_{3m}$ is prime and every induced subgraph of $C_{3m}$ is recolorable, while Lemma \ref{$H_m$} gives a non-recolorable blowup $H_m$ of $C_{3m}$. Therefore, there are infinitely many prime graphs $G$ such that every induced subgraph of $G$ is recolorable but some blowup of $G$ is not recolorable.
\end{proof}

\section{Proof of Theorem \ref{thm:lifting}}
Let $p\leq r-1$ and fix a proper $p$-coloring of $H(a)$. Let $I_1, \dots, I_r \subseteq S$ be the supports of the frozen coloring of
$H(a)$, clearly, they are MISs of $H$ by Lemma~\ref{frozen}, and
\begin{equation}\label{eq:coverage}
  a(v) = |\{ i \in [r] : v \in I_i \}| \qquad (v \in S).
\end{equation}

\textbf{Step 1: Choose the enlarged core.}
Partition the outside vertices into
\begin{align*}
  Z &= \{\, x \in V(G) \setminus S : N_G(x) \cap I_i \neq \varnothing
        \text{ for all } i \in [r] \,\},\\
  Y &= V(G) \setminus (S \cup Z).
\end{align*}
Write $Z = \{z_1, \dots, z_m\}$ and $Y = \{y_1, \dots, y_q\}$, and put
$T = S \cup Z$. Each $I_i$ remains a MIS of $G[T]$: it dominates $S$ by its
maximality in $H$ and dominates $Z$ by definition. Crucially, the original
supports are not extended into $Y$. For each $j \in [m]$, choose a MIS $J_j$ of $G[T]$ containing $z_j$. This is
constructive: fix a vertex ordering, extend $\{z_j\}$ greedily to a MIS $J_j$ of $G[T]$.
Define
\[
  d(v) = |\{ j \in [m] : v \in J_j \}| \quad (v \in T),
  \qquad
  t = q + 1.
\]
In particular, $d(z_j) \geq 1$.

\medskip
\textbf{Step 2: Specify every weight.}
Set
\begin{equation}\label{eq:weights}
  w(v) =
  \begin{cases}
    t\,a(v) + d(v), & v \in S,\\
    d(v),           & v \in Z,\\
    1,              & v \in Y,
  \end{cases}
  \qquad
  \ell = tr + m.
\end{equation}
Every weight is positive. Let $B = G(w)$, and let $F$ be its induced subgraph
on the replacement cliques corresponding to $T$.

For $v \in S$, split $Q_v$ into $t\,a(v)$ \emph{base} vertices and $d(v)$
\emph{auxiliary} vertices. For $z \in Z$, all vertices of $Q_z$ are auxiliary.
More precisely, there is a base vertex $b_{v,i,h}$ whenever $v \in I_i$ and
$h \in [t]$, and an auxiliary vertex $u_{v,j}$ whenever $v \in J_j$.
By \eqref{eq:coverage}, these counts agree with the weights in
\eqref{eq:weights}.

\medskip
\textbf{Step 3: Color and freeze the core.}
Use $tr$ distinct base colors $(i,h)$, where $i \in [r]$ and $h \in [t]$, and
$m$ further auxiliary colors $e_1, \dots, e_m$. Color $b_{v,i,h}$ with $(i,h)$ and $u_{v,j}$ with $e_j$.

The supports on $G[T]$ are $t$ copies of every $I_i$, together with
$J_1, \dots, J_m$. These supports are independent, so the coloring is proper; they are maximal,  so it is frozen on $F$ by Lemma~\ref{frozen}. Every color
occurs on $F$, since every support is nonempty.

\medskip
\textbf{Step 4: Extend to all outside singletons.}
For $y_h \in Y$, choose an index $i_h$ such that $N_G(y_h) \cap I_{i_h} = \varnothing$,
and give its single clone color $(i_h, h)$. This is available because
$h \leq q < t$. It creates no edge conflict with the core, because this base
color occurs there only on $I_{i_h}$. The outside colors are pairwise
distinct, since their second coordinates are distinct. Thus arbitrary edges
inside $Y$ also cause no conflict. We have constructed a proper $\ell$-coloring $\alpha$ of $B$. Its restriction
to $F$ is frozen, so by Lemma~\ref{lem:locking} every coloring reachable from $\alpha$ agrees with $\alpha$ on $F$.

\medskip
\textbf{Step 5: Exhibit a coloring below the obstruction level.}
Partition the base vertices into $t$ groups, each inducing a copy of $H(a)$.
Color these groups using $t$ disjoint-palette copies of the given proper
$p$-coloring. Edges between groups cause no conflict because their palettes
are disjoint. This uses $tp$ colors.

For each $j$, give all auxiliary vertices indexed by $J_j$ a single new
color. This uses $m$ colors and is proper because $J_j$ is independent. Give
the $q$ vertices over $Y$ distinct further colors. This explicit coloring
$\beta$ uses at most
$ tp + m + q \leq t(r-1) + m + q
  = tr + m - 1
  =\ell - 1,$
since $t = q + 1$. This proves the required chromatic-number bound.

\medskip
\textbf{Step 6: Separate two recoloring components.}
View $\beta$ as an $\ell$-coloring omitting one color. Every coloring
reachable from $\alpha$ agrees on $F$ with a coloring that uses all $\ell$
colors. Consequently, $\beta$ is not reachable from $\alpha$, and
$R_{\ell}(B)$ is disconnected.
\hfill$\qed$

\section{Proof of Theorem \ref{Induced-hole restriction}}
In this section, we give the proof of Theorem~\ref{Induced-hole restriction}. We first treat even holes separately. This case already illustrates the construction of a frozen obstruction on an induced cycle and its transfer to the ambient graph
via Theorem~\ref{thm:lifting}.

\begin{proposition}\label{prop:evenhole}
If $G$ contains an induced cycle $C_{2m}$ with $m \geq 3$, then $G$ has a
non-recolorable blowup. Explicitly, there are constructible positive weights
$w$ and an integer $\ell$ with $\chi(G(w)) \leq \ell - 1$ and
$R_{\ell}(G(w))$ disconnected.
\end{proposition}

\begin{proof}
Write $N = 2m \geq 6$ for the length of the induced cycle.

\medskip
\textbf{Step 1: A uniform frozen seed from cyclic shifts.}
Put $H = C_N$ and
\[
  \rho = \Bigl\lceil \frac{N}{3} \Bigr\rceil .
\]
We first construct an independent dominating set $I$ of cardinality $\rho$.
Let
$b_2 = 3\rho - N, b_3 = N - 2\rho$.
Both are nonnegative and satisfy
$b_2 + b_3 = \rho, 2b_2 + 3b_3 = N$.
Take a circular gap sequence consisting of $b_2$ gaps of length two and $b_3$
gaps of length three. Starting at $v_0$, take the vertices at successive
partial sums, omitting the final sum $N$. The selected set $I$ is independent
because every gap is at least two, and dominating because every gap is at
most three. Hence $I$ is a MIS.

Let $I_0, \dots, I_{N-1}$ be all $N$ cyclic shifts of $I$. Retain repeated
sets when different shifts coincide. Each vertex occurs in exactly $\rho$
shifts. Lemma~\ref{frozen} therefore gives a frozen $N$-coloring of the
uniform blowup $a(v) = \rho (v \in V(C_N))$. A proper $2\rho$-coloring uses one $\rho$-color palette on the even-indexed
replacement cliques and a disjoint palette on the odd-indexed cliques.

For even $N \geq 6$,
\begin{equation}\label{eq:seedbound}
  2\Bigl\lceil \frac{N}{3} \Bigr\rceil \leq N - 1.
\end{equation}
For completeness, write $N = 6s$, $6s+2$, or $6s+4$. The differences
$N - 2\rho$ are respectively $2s$, $2s$, and $2s$; the allowed values of $N$
make these at least two. In particular, the bound $\geq 1$ in
\eqref{eq:seedbound} suffices.

Thus the seed blowup $H(a)$ has a frozen $N$-coloring and a proper
$2\rho$-coloring with $2\rho \leq N - 1$. For example, $C_8$ with weight
three at every vertex has a frozen eight-coloring and a proper six-coloring.

\medskip
\textbf{Step 2: Explicit weights in the ambient graph.}
Suppose $G$ contains this induced $C_N$. Apply Theorem~\ref{thm:lifting} with
the shifted supports above: $r = N$, $a(v) = \rho$, and $p = 2\rho \leq r-1$.
Define $Z, Y, J_j, d$ as in Theorem~\ref{thm:lifting}, and write $m = |Z|$,
$q = |Y|$, and $t = q + 1$. Every weight is given by
\begin{equation}\label{eq:ambientweights}
  w(v) =
  \begin{cases}
    t\rho + d(v), & v \in V(C_N),\\
    d(v),         & v \in Z,\\
    1,            & v \in Y,
  \end{cases}
  \qquad
  \ell = tN + m.
\end{equation}
The $l$-coloring has a frozen core on $F$; the second coloring
uses $2\rho t + m + q$ colors, and $\ell - (2\rho t + m + q) = t(N - 2\rho) - q \geq t - q = 1$.
Therefore $\chi(G(w)) \leq \ell - 1$ and $R_{\ell}(G(w))$ is disconnected. This completes the proof.
\end{proof}

\begin{proof}[\bf Proof of Theorem \ref{Induced-hole restriction}]
It suffices to prove that for $N=6$ or $N\geq 8$, there exists a blowup of $G$ which is non-recolorable.
Put
\[
  \rho = \lceil N/3 \rceil,
  \qquad
  \sigma = \lfloor N/2 \rfloor .
\]
For the indicated values of $N$, we have $\rho < \sigma$: indeed $N = 6$
gives $2 < 3$, and for $N \geq 8$ one has
$\lceil N/3 \rceil \leq (N+2)/3 < (N-1)/2 \leq \lfloor N/2 \rfloor$.

\medskip
\textbf{Step 1: The frozen seed.}
The gap-two/gap-three construction of Proposition~\ref{prop:evenhole} also applies
for these (possibly odd) values of $N$: with $b_2 = 3\rho - N \geq 0$ and
$b_3 = N - 2\rho \geq 0$, it produces a MIS of size $\rho$. Take all $N$
cyclic shifts of this MIS and repeat each support $\sigma$ times. Each vertex lies in exactly $\rho\sigma$ of these supports, so the uniform seed weights $a(v) = \rho\sigma(v \in V(C_N))$
give a frozen $N\sigma$-coloring of the blowup $C_N(a)$, by
Lemma~\ref{frozen}.

\medskip
\textbf{Step 2: The low seed coloring.}
For a proper low coloring, take all $N$ cyclic shifts of a maximum
independent set of size $\sigma$, repeating each one $\rho$ times. Every
vertex again has coverage $\rho\sigma$, and these independent supports give a
proper $N\rho$-coloring of the same blowup. Since $\rho < \sigma$, we have
$N\rho \leq N\sigma - 1$, so Theorem~\ref{thm:lifting} applies with
$r = N\sigma$ and $p = N\rho$.

\medskip
\textbf{Step 3: Explicit ambient weights.}
Construct $Z, Y, J_j, d$ from the frozen seed supports as in
Theorem~\ref{thm:lifting}, writing $m = |Z|$, $q = |Y|$, $t = q + 1$, the
ambient weights are
\begin{equation*}\label{eq:mainweights}
  w(v) =
  \begin{cases}
    t\rho\sigma + d(v), & v \in V(C_N),\\
    d(v),               & v \in Z,\\
    1,                  & v \in Y,
  \end{cases}
  \qquad
  \ell = tN\sigma + m.
\end{equation*}
The high coloring has a permanently frozen induced core. The low coloring
uses at most $tN\rho + m + q$ colors, and
  $\ell - (tN\rho + m + q)
  = tN(\sigma - \rho) - q
  \geq t - q
  = 1.$
Hence $\chi(G(w)) \leq \ell - 1$, and the permanently frozen core establishes
that $R_{\ell}(G(w))$ is disconnected. This completes the proof.
\end{proof}

\section{Proof of Theorem \ref{thm:3457}}
As shown in the introduction, it suffices to prove that every blowup of $C_7$ is recolorable.
\begin{theorem}\label{thm:c7}
    For every choice of positive integers \(w_0,\ldots,w_6\), the clique blowup \(G=C_7(w_0,\ldots,w_6)\) is recolorable. Equivalently,
\[
R_k(G)\text{ is connected for every }k\ge \chi(G)+1.
\]
\end{theorem} 

For the proof, we work in the slightly larger class \(\mathcal C\) obtained by allowing some replacement cliques \(Q_i\) to be empty. This class is closed under taking induced subgraphs. In any proper coloring of a graph in \(\mathcal C\), a color occurs at most once in each \(Q_i\). Its support is the set of positions where it occurs. Every support is an independent set of \(C_7\), so every color class has size at most three.

We abbreviate supports by concatenating their positions: for example, \(02=\{0,2\}\) and \(146=\{1,4,6\}\). Different colors can have the same support. A move transferring position \(i\) from one support to another means recoloring the corresponding  vertex of $G$. Such a move is legal when the receiving support does not already contain \(i\), and the resulting support remains independent.

\begin{lemma}
    Let \(G\in\mathcal C\), and let \(\alpha\) be a proper \(k\)-coloring such that:
\begin{enumerate}
\item \(k>\omega(G)\);
\item every palette color is used on two or three vertices;
\item at least two color classes have size two.
\end{enumerate}
Then a coloring with a singleton color class is reachable from \(\alpha\) in at most three single-vertex recoloring moves.
\end{lemma}

\begin{proof}
Call two size-two supports \emph{compatible} if no vertex can be transferred from either color class to the other. If two size-two classes have incompatible supports, such a transfer immediately produces a singleton. We may therefore assume that all size-two supports are pairwise compatible. For every cycle edge \(ij\), the union \(Q_i\cup Q_j\) is a clique, and $|Q_i|+|Q_j|\le \omega(G)<k$.
Thus some palette color is absent from this union. If every size-two support meets \(\{i,j\}\), that absent color must have a size-three support avoiding \(i,j\).

There is exactly one independent triple of \(C_7\) avoiding a specified edge: deleting the two endpoints leaves a five-vertex path, whose unique independent set of size three consists of its first, third, and fifth vertices.

\textbf{Case 1: Two size-two classes have different supports}

Up to rotation, reflection, and interchanging the selected classes, their supports form one of the five pairs in the table below.

We verify that the cases are exhaustive. An independent pair of cycle positions has cyclic distance two or three. If one selected pair has distance two, normalize it to \(02\). Its distinct compatible partners are precisely
\[
03,\quad 13,\quad 16,\quad 26,\quad 36.
\]
The reflection \(i\mapsto 2-i\pmod 7\) fixes \(02\), interchanges \(03\) with \(26\), and interchanges \(13\) with \(16\). This leaves the first three cases below.

If both selected pairs have distance three, normalize one to \(03\). Its distinct compatible distance-three partners are precisely
\[
04,\quad 14,\quad 26,\quad 36.
\]
The reflection \(i\mapsto 3-i\pmod 7\) fixes \(03\), interchanges \(04\) with \(36\), and interchanges \(14\) with \(26\). This gives the last two cases.

For each row, the second column lists all independent pairs compatible with both selected supports. Every size-two color class in the current coloring must have one of these supports. All listed supports meet the displayed edge, so the preceding observation forces the existence of a color class with the displayed triple support.

\begin{center}
\begin{tabular}{@{}llll@{}}
\toprule
Selected supports & All possible size-two supports & Edge met by all & Forced triple \\
\midrule
\(02,\ 03\) & \(02,03,13,26,36\) & \(23\) & \(146\) \\
\(02,\ 13\) & \(02,03,13\) & \(23\) & \(146\) \\
\(02,\ 36\) & \(02,03,26,36\) & \(60\) & \(135\) \\
\(03,\ 04\) & \(03,04,13,14,36,46\) & \(34\) & \(025\) \\
\(03,\ 14\) & \(03,04,13,14\) & \(34\) & \(025\) \\
\bottomrule
\end{tabular}
\end{center}

The entries in the second column follow directly from the definition of compatibility: a position of one pair can be added to another pair exactly when it is absent from that pair and nonadjacent to both of its positions.

For the five rows, respectively, perform the following recolorings:
\[
\begin{aligned}
(02,03,146)
&\longrightarrow (024,03,16)
 \longrightarrow (024,0,136),\\
(02,13,146)
&\longrightarrow (024,13,16)
 \longrightarrow (024,1,136),\\
(02,36,135)
&\longrightarrow (025,36,13)
 \longrightarrow (025,136,3),\\
(03,04,025)
&\longrightarrow (035,04,02)
 \longrightarrow (035,024,0),\\
(03,14,025)
&\longrightarrow (035,14,02)
 \longrightarrow (035,1,024).
\end{aligned}
\]

Each coordinate corresponds to a fixed color, and all unlisted colors remain unchanged. Every arrow transfers exactly one position between two coordinates. The receiving support remains independent, so each arrow is a legal single-vertex recoloring. Each final coloring has a singleton.

\textbf{Case 2: All size-two classes have the same support}

By symmetry, their common support is either \(02\) or \(03\). By the lemma's hypothesis, choose two distinct colors with this support.

If the common support is \(02\), every size-two support meets both edges \(12\) and \(23\). The observation above forces colors with supports \(035\) and \(146\), respectively. Using these four distinct colors, perform
\[
\begin{aligned}
(02,02,035,146)
&\longrightarrow (025,02,03,146)\\
&\longrightarrow (025,024,03,16)\\
&\longrightarrow (025,024,0,136).
\end{aligned}
\]

If the common support is \(03\), edges \(34\) and \(23\) force colors with supports \(025\) and \(146\), respectively. Perform
\[
\begin{aligned}
(03,03,025,146)
&\longrightarrow (035,03,02,146)\\
&\longrightarrow (035,03,024,16)\\
&\longrightarrow (035,0,024,136).
\end{aligned}
\]

Again, every arrow is a legal single-vertex recoloring, and each final coloring has a singleton. This proves the lemma. 
\end{proof} 

\begin{proof}[\bf Proof of Theorem~\ref{thm:c7}]
We prove, by induction on the number of vertices, that every graph in \(\mathcal C\) is recolorable. The empty graph is trivial. Let \(G\in\mathcal C\) be nonempty, and assume the statement holds for all smaller graphs in \(\mathcal C\). Put \(p=\chi(G)\), fix a proper \(p\)-coloring \(\beta\), and let \(k\ge p+1\). We regard \(\beta\) as a coloring from the palette \(\{1,\ldots,k\}\). We will connect every proper \(k\)-coloring \(\alpha\) to this fixed coloring \(\beta\).

\textbf{Case 1: Reduction using an unused or singleton color}

Suppose that a color \(c\) is unused or occurs on exactly one vertex. If \(c\) occurs at a vertex \(v\), let \(S\) be the color class of \(\beta\) containing \(v\). If \(c\) is unused, choose any color class \(S\) of \(\beta\). Recolor the vertices of \(S\), one at a time, to \(c\). These moves are legal because \(S\) is independent and initially no vertex outside \(S\) has color \(c\). Keep \(S\) fixed and consider \(G-S\) with the remaining \(k-1\) colors. Since \(S\) is a color class of an optimal coloring,
\[
\chi(G-S)\le p-1.
\]
Also,
\[
k-1\ge p\ge \chi(G-S)+1.
\]
The induction hypothesis then implies that the recoloring graph of \(G-S\) on this remaining palette is connected. Consequently, we can recolor \(G-S\) to the partition into color classes prescribed by \(\beta\), assigning distinct available labels to its \(p-1\) classes. Every move remains legal in \(G\), because \(S\) uses only the excluded color \(c\).

We now have exactly the color-class partition of \(\beta\), possibly with different labels. Because \(k>p\), at least one color is unused. The labels can be corrected using this unused color as a buffer: transferring an entire independent class to an unused color is legal one vertex at a time, and a buffer permits any permutation of the class labels. Hence the current coloring can be recolored to \(\beta\). Thus any coloring with an unused or singleton color is connected to \(\beta\).

\textbf{Case 2: Producing a singleton in every remaining coloring}

Now suppose \(\alpha\) has neither an unused color nor a singleton color class. Every color class then has size two or three, since every independent set in \(G\) has size at most three.
Let \(n=|V(G)|\), and let \(q\) be the number of size-two color classes. All \(k\) colors are used, so $n=2q+3(k-q)=3k-q$. The optimal coloring \(\beta\) partitions \(V(G)\) into \(p\) independent sets, each of size at most three. Thus \(n\le 3p\), and $q=3k-n\ge 3(k-p)\ge 3$.
Moreover,
\[
k>p=\chi(G)\ge \omega(G).
\]
All hypotheses of the lemma hold. Therefore \(\alpha\) can be recolored, in at most three moves, to a coloring with a singleton color class. The preceding reduction then connects it to \(\beta\).

Every proper \(k\)-coloring is therefore connected to the same fixed coloring \(\beta\), so \(R_k(G)\) is connected. This holds for every \(k\ge\chi(G)+1\), completing the induction. In particular, restricting to nonempty replacement cliques yields the theorem.
\end{proof}

\section{Proof of Theorem \ref{fractionally
co-strongly perfect}}
In this section, we prove Theorem~\ref{fractionally co-strongly perfect}. The key step is Proposition~\ref{proposition1.7} below, which reformulates the existence of a saturating vertex weighting in terms of frozen colorings of blowups. The reformulation is natural once one observes that a saturating vertex weighting is precisely a nonnegative solution to the linear system
$$\sum y(v) = 1 \qquad\text{for every maximal independent set $I$ of $H$, $v\in I$},$$
so its existence is a linear feasibility problem over the incidence vectors of the maximal independent sets. Linear programming duality then provides a certificate whenever no such solution exists, and the dual certificate turns out to yield precisely the configuration required by Theorem~\ref{thm:lifting}: a blowup of $H$ whose frozen coloring is obtained from the maximal independent sets themselves, together with a proper coloring with fewer colors.

\begin{proposition}\label{proposition1.7}
For a nonempty graph $H$, the following are equivalent:
\begin{enumerate}
\item[\textup{(i)}] $H$ has a saturating vertex weighting.
\item[\textup{(ii)}] No positive-weight blowup of $H$ has a frozen
$r$-coloring and a proper coloring with at most $r-1$ colors.
\end{enumerate}
\end{proposition}
\begin{proof}
$(i)\Rightarrow (ii).$ Suppose to the contrary that there exists a blowup $H(w)$ of $H$, such that $H(w)$ has a frozen $r$-coloring with maximal independent sets $I_1,\dots I_r$ by Lemma \ref{frozen} and a proper coloring with at most $r-1$ colors. Let $y$ be the saturating vertex weighting of $H$. Counting color occurrences at each replacement clique gives
\begin{equation}\label{6.1}
    \sum_{v\in V(H)}y(v)w(v)=\sum_{i=1}^{r}\sum_{v\in I_i}y(v)=r.
\end{equation}
Note that $H(w)$ also has a proper $p$-coloring, where $p\leq r-1$. Clearly, each of its $p$ color supports $\{J_1,\dots, J_p\}$ is independent, then the same counting gives
\begin{equation}\label{6.2}
\sum_{v\in V(H)}y(v)w(v)=\sum_{i=1}^{p}\sum_{v\in J_i}y(v)\leq p.
\end{equation}
Combining \eqref{6.1} and \eqref{6.2} gives $p\geq r$, a contradiction.

$(ii)\Rightarrow (i).$ Let $I_1,\dots, I_t$ be all the maximal independent sets of $H$. It is straightforward to check that $V(H)=\bigcup_{i=1}^t I_i$. Let $A$ be the matrix whose rows are indexed by vertices of $H$ and whose columns are indexed by $I_1,\dots, I_t$, with $A_{v, i}=1$ if $v\in I_i$ and 0 otherwise.
Let $A\mathbf{1}=a$, then $a(v)=|\{i:v\in I_i\}|$. Moreover, each entry of $a$ is positive since every vertex of $H$ must belong to a set of $\{I_1,\dots, I_t\}$.

From the above discussion, consider the primal and dual linear programs
\begin{equation}\label{6.3}
\begin{aligned}
    \min \quad  &\mathbf{1}^\mathsf{T}x\\
\text{s.t.} \quad  &Ax\geq a,\\
&x\geq0,
\end{aligned}
\end{equation}
and
\begin{equation}\label{6.4}
\begin{aligned}
    \max \quad  &a^\mathsf{T}y\\
\text{s.t.} \quad  &A^\mathsf{T}y\leq \mathbf{1},\\
&y\geq0.
\end{aligned}
\end{equation}
Note that the optimum value is finite: the vector $x=\mathbf{1}$ is primal-feasible with the objective value $t$, while the vector $y=\mathbf{0}$ is dual-feasible with value $0$. Therefore, by strong duality the two optima are equal and attained. If the optimum is $t$, an optimal dual vector satisfies
$$t=a^\mathsf{T}y=\mathbf{1}^\mathsf{T}A^\mathsf{T}y.$$
There are $t$ entries in $A^\mathsf{T}y$, each of which is at most one. Thus every entry of $A^\mathsf{T}y$ must equal one; it also implies that $y$ is a saturating vertex weighting of $H$.
Consequently, if no saturating vertex weighting exists, the optimum in \eqref{6.3} is strictly smaller than $t$. Since the data are rational, we can choose a rational optimal solution $x$ with 
$$\mathbf{1}^\mathsf{T}x<t.$$
Let $D$ be a positive integer clearing the denominators of all entries of $x$. Define every clique weight by 
\begin{equation*}
    w(v)=Da(v)=D|\{i:v\in I_i\}|\quad (v\in V(H)).
\end{equation*}
Observe that each $w(v)$ is a positive integer. Repeat each maximal independent set $I_i$ exactly $D$ times, the coverage at $v$ is $Da(v)=w(v)$. By the proof of Corollary \ref{coro_frozen} we obtain a frozen $Dt$-coloring of $H(w)$. 

We now construct a proper coloring of $H(w)$ with fewer colors. Take $Dx_i$ distinct colors with support $I_i$. 
The definition of $D$ guarantees that each $Dx_i(i\in [t])$ is an integer. The feasibility inequality gives 
$$A(Dx)\geq Da=w.$$
Therefore, for each vertex $v$ of $H$, there are at least $w(v)$ colors available. Then we can choose $w(v)$ colors for the copies of $v$ in $H(w)$. Each retained color support is an independent set of $H$, so the coloring is proper. Note that the total number of colors used is at most
$$D(x_1+\dots+x_t)=D\mathbf{1}^\mathsf{T}x\leq Dt-1.$$
As a consequence, $H(w)$ has both a frozen $Dt$-coloring and a proper coloring with at most $Dt-1$ colors, a contradiction. Thus $H$ has a saturating vertex weighting.
\end{proof}

Now we are ready to give the proof of Theorem \ref{fractionally
co-strongly perfect}.

\begin{proof}[\bf Proof of Theorem \ref{fractionally
co-strongly perfect}]
Suppose to the contrary that there exists an induced subgraph $H$ of $G$ which has no saturating vertex weighting. According to Proposition \ref{proposition1.7}, there exists a positive-weight blowup of $H$ with a frozen coloring above its chromatic number. Theorem \ref{thm:lifting} then gives a non-recolorable blowup of $G$, which leads to a contradiction.
\end{proof}

\section{Conclusion}
In this paper, we answer Question \ref{Question-BCS} negatively and disprove Conjecture \ref{Conjecture-BCS} via a frozen-coloring framework for blowups, which further yields structural restrictions on graphs whose every blowup is recolorable. The remaining open problem is to characterize the graphs $G$ for which every blowup of $G$ is recolorable. We believe that this unique structural restriction will lead to the final outcome.

\section*{Acknowledgements}
Tianyu Huang, Yulai Ma, and Jiaqi Zhang were partially supported by the National Natural Science Foundation of China (No.
12601673, 12631013).

\subsection*{Declaration of generative AI use}
During the preparation of this work, the authors used ChatGPT to improve the language and readability of the manuscript, 
including grammatical corrections and language polishing. All mathematical results and proofs were developed and verified by the 
authors, who take full responsibility for the content of the article.

\bibliographystyle{plain}
\bibliography{bib}

\end{document}